\documentclass[12pt]{amsart}

\usepackage[bookmarks=false,unicode,colorlinks,urlcolor=red,citecolor=red,linkcolor=blue]{hyperref}
\usepackage{verbatim}
\usepackage{fullpage}
\usepackage{amsfonts}
\usepackage{amssymb}

\renewcommand{\tilde}{\widetilde}

\newcommand{\R}{\mathbb{R}}
\renewcommand{\C}{\mathbb{C}}
\newcommand{\N}{\mathbb{N}}

\newcommand{\Hil}{\mathcal{H}}

\newcommand{\la}{\lambda}

\DeclareMathOperator{\lspan}{span}

\DeclareMathOperator{\diag}{diag}

\newcounter{Theorem}

\numberwithin{equation}{section}
\numberwithin{Theorem}{section}

\theoremstyle{plain} 
\newtheorem{thm}[Theorem]{Theorem}

\newtheorem{lem}[Theorem]{Lemma}

\theoremstyle{definition}
\newtheorem{defn}[Theorem]{Definition}

\theoremstyle{remark}

\newcommand{\mref}[1]{%
\href{http://www.ams.org/mathscinet-getitem?mr=#1}{#1}}

\newcommand{\arxiv}[1]{%
\href{http://front.math.ucdavis.edu/#1}{ArXiv:#1}}

\newcommand{\zbl}[1]{%
\href{http://zbmath.org/?q=an:#1}{Zbl #1}}

\begin{document}

\title{The Schur-Horn theorem for unbounded operators with discrete spectrum}

\author{Marcin Bownik}
\address{Department of Mathematics, University of Oregon, Eugene, OR 97403--1222, USA}
\email{mbownik@uoregon.edu}

\author{John Jasper}
\address{Department of Mathematical Sciences,
University of Cincinnati,
Cincinnati, OH 45221--0025, USA}
\email{jasperjh@ucmail.uc.edu}

\author{Bart{\l}omiej Siudeja}
\address{Department of Mathematics, University of Oregon, Eugene, OR 97403--1222, USA}
\email{siudeja@uoregon.edu}


\date{\today}

\begin{abstract}
We characterize diagonals of unbounded self-adjoint operators on a Hilbert space $\mathcal H$ that have only discrete spectrum, i.e., with empty essential spectrum. Our result extends the Schur-Horn theorem from a finite dimensional setting to an infinite dimensional Hilbert space, analogous to Kadison's theorem for orthogonal projections \cite{k1,k2}, Kaftal and Weiss \cite{kw} results for positive compact operators, and Bownik and Jasper \cite{mbjj2, mbjj3, jj} characterization for operators with finite spectrum. Furthermore, we show that if a symmetric unbounded operator $E$ on $\mathcal H$ has a nondecreasing unbounded diagonal, then any sequence that weakly majorizes this diagonal is also a diagonal of $E$.
\end{abstract}

\thanks{The first author was partially supported by NSF grant DMS-1265711. The third author was partially supported by NCN grant 2012/07/B/ST1/03356.}

\maketitle

\section{Introduction}

The classical Schur-Horn theorem characterizes diagonals of hermitian matrices in terms of their eigenvalues. An infinite dimensional extensions of this result has been a subject of intensive study in recent years. This line of research was jumpstarted by the influential work of Kadison \cite{k1, k2}, who discovered a characterization of diagonals of orthogonal projections on separable Hilbert space, and by Arveson and Kadison \cite{ak} who extended the Schur-Horn theorem to positive trace class operators. This has been preceded by earlier work of Gohberg and Markus \cite{gm} and by
Neumann \cite{neu}. The Schur-Horn theorem has been extended to compact positive operators by Kaftal and Weiss \cite{kw} and Loreaux and Weiss \cite{lw} in terms of majorization inequalities \cite{kw0}. Lebesgue type majorization was used by Bownik and Jasper \cite{mbjj2, mbjj3, jj} to characterize diagonals of self-adjoint operators with finite spectrum operators. Other notable progress includes the work of Arveson \cite{a} on diagonals of normal operators with finite spectrum and 
Antezana, Massey, Ruiz, and Stojanoff's results \cite{amrs}. Finally, there is a rapidly growing body of literature on the corresponding problems for von Neumann algebras \cite{am, am2, am3, dfhs, ks, rara, ravi}.

The goal of this paper is to prove an infinite dimensional variant of the Schur-Horn theorem for unbounded self-adjoint operators with discrete spectrum. This represents a new direction in extending the Schur-Horn theorem to infinite dimensional setting since previous results dealt only with bounded operators.

Assume that an unbounded self-adjoint operator $E$ on a separable Hilbert space $\mathcal H$ is bounded from below and has discrete spectrum. That is, the essential spectrum $\sigma_{ess}(E) = \emptyset$, and hence, every point $\lambda \in \sigma(E)$ is an isolated eigenvalue of finite multiplicity. Since $E$ is bounded from below, its eigenvalues can be listed by a nondecreasing sequence $\boldsymbol\lambda = \{\lambda_{i}\}_{i\in \N}$ according to their multiplicities. Since $\sigma_{ess}(E) = \emptyset$, we must necessarily have $\lim_{i\to \infty} \lambda_i =\infty$, and thus $E$ is unbounded from above. 

Consequently, $E$ is diagonalizable, i.e., there exists an orthonormal basis $\{v_i\}_{i\in \N}$ of eigenvectors $Ev_{i}=\lambda_{i}v_{i}$ for all $i\in \N$, and the domain of $E$ is given by
\begin{equation}\label{domain}
\mathcal D = \bigg\{ f\in \mathcal H:  \sum_{i\in \N} |\la_i|^2 |\langle f, v_i \rangle |^2 < \infty \bigg\}.
\end{equation}
In order to emphasize this point we will use the notation $E=\diag\boldsymbol\lambda$ to denote the operator which has eigenvalues $\boldsymbol\lambda$ and domain \eqref{domain} as above.

If $\{e_i\}_{i\in \N} \subset \mathcal D$ is any other orthonormal basis of $\mathcal H$, then the diagonal $d_i= \langle Ee_i, e_i\rangle$ of $E$ with respect to $\{e_{i}\}$ satisfies 
\begin{equation}\label{maj}
\sum_{i=1}^{n}\la_{i} \leq \sum_{i=1}^{n}d_{i} \qquad\text{for all } n\in \N.
\end{equation}
In particular, the same inequality holds true when $\{d_i\}_{i\in \N}$ is replaced by its increasing rearrangement $\{d_i^\uparrow\}_{i\in\N}$. The necessity of condition \eqref{maj} is often attributed to Schur \cite{schur}. Our main result says that \eqref{maj} is also sufficient, thus generalizing Horn's theorem \cite{horn}.

\begin{thm}\label{unb}
Suppose that $\boldsymbol\lambda=\{\la_i\}_{i\in\N}$ and $\{d_i\}_{i\in\N}$ are two nondecreasing and unbounded sequences. Let $E=\diag\boldsymbol\lambda$ be a self-adjoint operator with eigenvalues $\boldsymbol \lambda$ and eigenvectors $\{v_i\}_{i\in \N}$. If the majorization inequality \eqref{maj} holds, then there exists an orthonormal basis $\{e_{i}\}_{i\in\N}$, which lies in the linear span of $\{v_i\}_{i\in \N}$, such that $d_{i} = \langle E e_{i},e_{i}\rangle$ for all $i\in\N$.
\end{thm}

The remarkable consequence of our main theorem is that majorization inequality \eqref{maj} is the only condition that a sequence $\{d_i\}_{i\in \N}$ must satisfy in order to be diagonal of $\diag\boldsymbol\lambda$. Moreover, the required diagonal is achieved with respect to an o.n.~basis $\{e_i\}_{i\in\N}$, whose elements are finite linear combinations of eigenvectors $\{v_i\}_{i\in \N}$. In particular, it is possible that $\lambda_i=d_i$ for all but finitely many $i\in\N$, the trace condition is violated, i.e., $\sum_{i=1}^\infty (d_i-\lambda_i) \ne 0$, but yet the conclusion of Theorem \ref{unb} still holds. 

Despite the simplicity of the statement of Theorem \ref{unb}, its proof is far from trivial as it needs to deal with two major cases. The majorization inequality \eqref{maj} can be equivalently stated as
\[
\delta_{k} = \sum_{i=1}^{k}(d_{i}-\lambda_{i})\geq 0\quad\text{for all }k\in\N.
\]
After dealing with elementary reductions in Section \ref{S2}, 
the first case deals with the conservation of mass scenario
\[
\liminf_{k\to\infty} \delta_k =0.
\]
The second case deals with vanishing mass at infinity scenario 
\[
\alpha=\liminf_{k\to\infty} \delta_k >0.
\]
This further splits in two subcases: $\delta_k \ge \alpha$ for sufficiently large $k$, and $\delta_k<\alpha$ for infinitely many $k$, shown by Theorems \ref{limalpha} and \ref{tonondec}, respectively. 

The proofs of these cases require careful application of an infinite sequence of convex moves, also known as $T$-transforms \cite{kw}, to guarantee that the limiting o.n.~sequence is a basis. In addition, we need to ensure that the constructed basis is contained in the dense domain $\mathcal D$. This constraint was not present in earlier work on bounded operators and requires new techniques of moving from a prescribed diagonal into a desired diagonal configuration. Our methods work not only for self-adjoint operators with discrete spectrum as in Theorem \ref{unb}, but also for unbounded symmetric operators (possibly with continuous spectrum) as in Theorem \ref{d2d}. Indeed, ``eigenvalue to diagonal'' Theorem \ref{unb} is  an immediate consequence of a more general ``diagonal to diagonal'' Theorem \ref{d2d}.

We end the paper by giving several examples illustrating Theorem \ref{unb} in Section \ref{S5}. Laplacians, or more generally elliptic differential operators, provide a broad and interesting class of operators falling into the scope of this paper.

\section{Diagonal to diagonal elementary reductions}\label{S2}

In this section we show several reductions that are employed in the proof of Theorem \ref{unb}. To achieve this we formulate a generalization of Theorem \ref{unb} for unbounded symmetric operators which are not necessarily diagonalizable. Recall that a linear operator $E$ defined on a dense domain $\mathcal D \subset \mathcal H$ is {\it symmetric} if
\[
\langle Ef,g \rangle =\langle f, Eg \rangle \qquad\text{for all }f,g \in \mathcal D.
\]
Theorem \ref{unb} is an immediate consequence of the following diagonal to diagonal theorem.

\begin{thm}\label{d2d}
Let $E$ be a symmetric operator defined on a dense domain $\mathcal{D} \subset \mathcal H$.  Let $\mathbf d=\{d_{i}\}_{i\in\N}$ and $\boldsymbol\lambda=\{\lambda_{i}\}_{i\in\N}$ be two nondecreasing unbounded sequences satisfying \eqref{maj}. If there exists an orthonormal sequence $\{f_{i}\}_{i\in\N}\subset\mathcal{D}$ such that 
\[
\langle Ef_{i},f_{i}\rangle = \lambda_i
\qquad\text{for all } i\in \N,
\]
then there exists an orthonormal sequence $\{e_{i}\}_{i\in\N}\subset\lspan\{f_{i}\}_{i\in\N}$ such that $\overline{\lspan}\{e_{i}\}_{i\in\N}=\overline{\lspan}\{f_{i}\}_{i\in\N}$ and 
\[
\langle Ee_{i},e_{i}\rangle = d_{i}\qquad\text{for all }i\in\N.
\]
\end{thm}

In the special case when $\{f_i\}_{i\in\N}$ is an orthonormal basis of eigenvectors with eigenvalues $\{\lambda_i\}_{i\in\N}$ of a self-adjoint operator $E=\diag\boldsymbol\lambda$, Theorem \ref{d2d} immediately yields Theorem \ref{unb}. 
To facilitate statements of reduction results, we shall make some formal definitions.

\begin{defn}\label{d21}
 Let $\boldsymbol\lambda = \{\lambda_{i}\}_{i\in I}$ and $\mathbf d=\{d_i\}_{i\in I}$ be two real sequences, where $I$ is countable. Let $E$ be unbounded (here it means not necessarily bounded) linear operator defined on a dense domain $\mathcal D$ of a Hilbert space $\mathcal H$. We say that an operator $E$ has {\it diagonal} $\boldsymbol \lambda$ if there exists an orthonormal sequence $\{f_{i}\}_{i\in I}$ contained in $\mathcal D$ such that 
 \[
\langle Ef_{i},f_{i}\rangle = \lambda_{i}\qquad\text{for all }i\in I.
\]

We say that $E$ has diagonal $\mathbf d$, which is {\it finitely derived} from diagonal $\boldsymbol\lambda$, if there
exists an orthonormal sequence $\{e_{i}\}_{i\in I}$ in $\mathcal D$ satisfying
$\langle Ee_{i},e_{i}\rangle = d_{i}$ for all $i\in I$, 
\begin{equation}\label{block2}
\overline{\lspan}\{e_{i}\}_{i\in\N}=\overline{\lspan}\{f_{i}\}_{i\in\N}
\qquad\text{and}\qquad  \forall k\in I \quad e_k \in \lspan \{f_i\}_{i\in I}.
\end{equation}
\end{defn}

Suppose $\{\lambda_i\}_{i=1}^N$ and $\{d_i\}_{i=1}^N$ are two real sequences. Let $\{\lambda_i^\downarrow\}_{i=1}^N$ and $\{d_i^\downarrow\}_{i=1}^N$ be their decreasing rearrangements. Following \cite{moa} we define a majorization order $\{d_i\} \preccurlyeq \{\lambda_i\}$ if and only if
\begin{equation}\label{horn1}
\sum_{i=1}^{N}d^\downarrow_i =\sum_{i=1}^{N}\lambda^\downarrow_{i} \quad\text{ and }\quad \sum_{i=1}^{n}d^\downarrow_{i} \leq \sum_{i=1}^{n}\lambda^\downarrow_{i} \quad\text{for all } 1\le n \le N.
\end{equation}

The classical Schur-Horn theorem \cite{horn, schur} characterizes diagonals of self-adjoint (Hermitian) matrices with given eigenvalues. It can be stated as follows, where $\mathcal H_N$ is an $N$ dimensional Hilbert space over $\R$ or $\C$, i.e., $\mathcal H_N=\R^N$ or $\C^N$.

\begin{thm}[Schur-Horn theorem]\label{horn} 
There exists a self-adjoint operator $E:\mathcal H_N \to\mathcal H_N$ with eigenvalues $\{\lambda_{i}\}_{i=1}^N$ and diagonal $\{d_{i}\}_{i=1}^N$
if and only if $\{d_i\} \preccurlyeq \{\lambda_i\}$.
\end{thm}

As a consequence of Theorem \ref{horn} we have the following block diagonal lemma.

\begin{lem}\label{block}
Let $E$ be a symmetric operator defined on a dense domain $\mathcal D \subset \mathcal H$. Suppose that $\{d_i\}_{i\in I}$ and $\{\tilde d_{i}\}_{i\in I}$ are two sequence of real numbers such that:
\begin{enumerate}
\item there exists a collection of disjoint finite subsets $\{I_{j}\}_{j\in J}$ of the index set $I$,
\item  $\{d_i\}_{i\in I_j}\preccurlyeq \{\tilde d_i\}_{i\in I_j}$ for each $j\in J$,
\item $\tilde{d}_{i}=d_{i}$ for all $i\in I\setminus\left(\bigcup_{j\in J}I_{j}\right)$.
\end{enumerate}
Suppose that $E$ has diagonal $\{\tilde d_{i}\}_{i\in I}$ with respect to an orthonormal sequence $\{f_{i}\}_{i\in I}$.
Then, $\{d_{i}\}_{i\in I}$ is a finitely derived diagonal of $E$. That is, there exists an orthonormal sequence $\{e_{i}\}_{i\in I}$ satisfying \eqref{block2}
with respect to which $E$ has diagonal $\{d_{i}\}_{i\in I}$.
\end{lem}

\begin{proof} 
Let $P_j$ be the orthogonal projection of $\mathcal H$ onto finite dimensional block subspace $\mathcal H_j = \operatorname{span}\{f_i: i \in I_j\}$. Observe that a finite dimensional self-adjoint operator $E_j:=(P_j E)|_{\mathcal H_j}$ has diagonal $\{\tilde d_i\}_{i\in I_j}$ with respect to $\{f_i\}_{i\in I_j}$. By (ii) and the Schur-Horn Theorem, there exists a unitary operator $U_j$ on $\mathcal H_j$ such that $U_j E_j (U_j)^*$ has diagonal $\{d_i\}_{i\in I_j}$ with respect to $\{f_i\}_{i\in I_j}$. Define an o.n.~basis $\{e_i\}_{i\in I}$ by
\[
e_{i}= \begin{cases} U_j f_j & i \in I_j,
\\
f_i & i\in I\setminus\left(\bigcup_{j\in J}I_{j}\right).
\end{cases}
\]
For $i\in I_j$ we have
\[
\langle E e_i, e_i \rangle = \langle E (U_j)^* f_i, (U_j)^* f_i \rangle 
= \langle P_j E (U_j)^* f_i, (U_j)^* f_i \rangle 
= \langle U_j E_j (U_j)^* f_i, f_i \rangle =d_i.
\]
The same identity holds trivially for $i\not\in \bigcup_{j\in J}I_{j}$, which shows that $E$ has diagonal $\{d_i\}$ with respect to $\{e_i\}$.
This completes the proof of the lemma.
\end{proof}

As an application of Lemma \ref{block} we can show the special case of Theorem \ref{d2d}.

\begin{lem}\label{infalphas} Let $E$ be a symmetric operator defined on a dense domain $\mathcal D \subset \mathcal H$.
Let $\mathbf d=\{d_{i}\}_{i\in\N}$ and $\boldsymbol\lambda=\{\lambda_{i}\}_{i\in\N}$ be two nondecreasing sequences such that
\begin{equation}\label{delta}
\delta_{k} := \sum_{i=1}^{k}(d_{i}-\lambda_{i})\geq 0\quad\text{for all }k\in\N.
\end{equation}
Suppose that there are infinitely many $k\in\N$ such that $\delta_{k} = 0$. If $\boldsymbol\lambda$ is diagonal of $E$, then $\mathbf d$ is a finitely derived diagonal of $E$.
\end{lem}

\begin{proof} Set $k_{1}=0$ and let $\{k_{j}\}_{j=2}^{\infty}$ be a strictly increasing sequence in $\N$ such that $\delta_{k_{j}}=0$ for all $j\geq 2$. For each $j\in\N$ set $I_{j} = \{k_{j}+1,\ldots,k_{j+1}\}$. For each $j\in\N$ and $k\in I_{j}$
\[\sum_{i=k_{j}+1}^{k}(d_{i}-\lambda_{i}) = \delta_{k}-\delta_{k_{j}} = \delta_{k}\geq 0.
\]
Since $\delta_{k_{j+1}}=0$ we have $\{d_{i}\}_{i\in I_{j}}\preccurlyeq \{\lambda_{i}\}_{i\in I_{j}}$. By our assumption, $E$ has diagonal $\boldsymbol \lambda$ with respect to some o.n.~sequence $\{f_i\}_{i\in \N}$. By Lemma \ref{block} there is an o.n.~sequence $\{e_i\}_{i\in \N}$ satisfying \eqref{block2} with respect to which $E$ has diagonal $\mathbf d$.
\end{proof}

In the proof of Theorem \ref{d2d} it is convenient to make the reducing assumption \eqref{pdelta} about nondecreasing sequences  $\boldsymbol \lambda$ and $\mathbf d$.

\begin{thm}\label{red}
 If Theorem \ref{d2d} holds under an additional assumption
\begin{equation}\label{pdelta}
\delta_k  = \sum_{i=1}^{k}(d_{i}-\lambda_{i}) >0 \qquad\text{for all }k\in \N,
\end{equation}
then it holds in a full generality.
\end{thm}

\begin{proof}
Suppose that $E$ has diagonal $\boldsymbol \lambda$ with respect to o.n.~sequence $\{f_i\}_{i\in\N}$.
The case when $\delta_k=0$ for infinitely many $k\in\N$ is covered by Lemma \ref{infalphas}. Hence, we can assume that there are finitely many $k\in\N$ such that $\delta_{k} = 0$. Let $N\in\N$ be the largest such integer. Define the spaces 
\[
\Hil_{0}= \overline{\operatorname{span}} \{f_{i}\}_{i=1}^{N}
\qquad\text{and}\qquad 
\Hil_{1}=\overline{\operatorname{span}} \{f_i\}_{i=N+1}^\infty.
\]
Applying Theorem \ref{d2d} to the sequences $\{d_{i}\}_{i=N+1}^{\infty}$ and $\{\lambda_{i}\}_{i=N+1}^{\infty}$, and noting that for $k\geq N+1$
\[
\sum_{i=N+1}^{k}(d_{i}-\lambda_{i}) = \delta_{k} - \delta_{N} = \delta_{k}>0
\]
we obtain an orthonormal basis $\{e_{i}\}_{i=N+1}^{\infty}$ of $\mathcal H_1$ such that $\langle E e_{i},e_{i}\rangle = d_{i}$ for all $i\geq N+1$. 
Then the operator $E$ has diagonal 
\[\lambda_{1},\ldots,\lambda_{N},d_{N+1},d_{N+2},\ldots\]
with respect to o.n.~basis $\{f_{1},\ldots,f_{N},e_{N+1},e_{N+2},\ldots\}$ of $\mathcal H_0 \oplus \mathcal H_1$.
Since $\delta_{N}=0$ we have $\{d_{i}\}_{i=1}^{N}\preccurlyeq \{\lambda_{i}\}_{i=1}^{N}$. Applying Lemma \ref{block} we obtain an o.n.~sequence $\{e_i\}_{i=1}^\infty$  with respect to which $E$ has diagonal $\mathbf d$ and \eqref{block2} holds.
\end{proof}

We end this section with a basic linear algebra lemma about convex moves of $2\times 2$ hermitian matrices. Lemma \ref{offdiag} generalizes the corresponding well-known result for matrices with zero off-diagonal entries.

\begin{lem}\label{offdiag} Let $E$ be a symmetric operator on $\mathcal D \subset \mathcal H$. Assume that real numbers $d_1$, $d_2$, $\tilde d_1$, $\tilde d_2$ satisfy
\begin{equation}\label{oft}
\tilde{d}_{1}\leq d_{1},d_{2} \leq \tilde{d}_{2},\qquad \tilde{d}_{1}\neq \tilde{d}_{2}, \qquad \text{and}\qquad \tilde{d}_{1}+\tilde{d}_{2}=d_{1}+d_{2}.
\end{equation}
 If there exists an orthonormal set $\{f_{1},f_{2}\} \subset \mathcal D$ such that $\langle Ef_{i},f_{i}\rangle = \tilde{d}_{i}$ for $i=1,2$, then there exists
\[
\frac{\tilde{d}_{2}-d_{1}}{\tilde{d}_{2}-\tilde{d}_{1}}
\le \alpha \le 1
\]
and $\theta\in[0,2\pi)$ such that $\langle Ee_{i},e_{i}\rangle = d_{i}$ for $i=1,2$, where
\begin{equation}\label{offdiag0}e_{1}=\sqrt{\alpha}f_{1} + \sqrt{1-\alpha}e^{i\theta}f_{2}\qquad\text{and}\qquad e_{2}=\sqrt{1-\alpha}f_{1}-\sqrt{\alpha}e^{i\theta}f_{2}.\end{equation}
Moreover, if $\mathcal H$ is a real Hilbert space, then $e^{i\theta} = \pm 1$. If the inequalities in \eqref{oft} are strict, then $\alpha<1$.
\end{lem}

\begin{proof} Set
\[\beta:=\langle Ef_{1},f_{2}\rangle.\]
Choose $\theta\in[0,2\pi)$ such that $e^{-i\theta}\beta\leq 0$. For $x\in[0,1]$ define 
\[e_{1}^{x}  = \sqrt{x}f_{1} + \sqrt{1-x}e^{i\theta}f_{2}\qquad \text{and} \qquad e_{2}^{x} = \sqrt{1-x}f_{1}-\sqrt{x}e^{i\theta}f_{2}.\]
We calculate
\[\langle Ee_{1}^{x},e_{1}^{x}\rangle = x\tilde{d}_{1} + (1-x)\tilde{d}_{2} + 2e^{-i\theta}\beta\sqrt{x(1-x)}\]
so that
\[\langle Ee_{1}^{1},e_{1}^{1}\rangle = \tilde{d}_{1}\geq d_{1}\]
and for $\alpha_{0} = (d_{1}-\tilde{d}_{2})/(\tilde{d}_{1}-\tilde{d}_{2})$, since $e^{-i\theta}\beta\leq 0$ we have
\begin{align*}
\langle Ee_{1}^{\alpha_{0}},e_{1}^{\alpha_{0}}\rangle & = \alpha_{0}(\tilde{d}_{1}-\tilde{d}_{2}) + \tilde{d}_{2} + 2e^{-i\theta}\beta\sqrt{\alpha_{0}(1-\alpha_{0})}\\
 & = d_{1}-\tilde{d}_{2} + \tilde{d}_{2} + 2e^{-i\theta}\beta\sqrt{\alpha_{0}(1-\alpha_{0})}\leq d_{1}.
\end{align*}
Since $x\mapsto\langle Ee_{1}^{x},e_{1}^{x}\rangle$ is continuous on $[\alpha_{0},1]$ there is some $\alpha\geq \alpha_{0}$ such that $\langle Ee_{1}^{\alpha},e_{1}^{\alpha}\rangle = d_{1}$. Finally, using the assumption that $\tilde{d}_{1}+\tilde{d}_{2}=d_{1}+d_{2}$, we have
\begin{align*}
\langle Ee_{2}^{\alpha},e_{2}^{\alpha}\rangle & = (1-\alpha)\tilde{d}_{1} + \alpha\tilde{d}_{2} - 2e^{-i\theta}\beta\sqrt{\alpha(1-\alpha)}\\
 & = \tilde{d}_{1}+\tilde{d}_{2} - \Big(\alpha\tilde{d}_{1}+(1-\alpha)\tilde{d}_{2} + 2e^{-i\theta}\beta\sqrt{\alpha(1-\alpha)}\Big)\\
 & = \tilde{d}_{1}+\tilde{d}_{2} - \langle Ee_{1}^{\alpha},e_{1}^{\alpha}\rangle = \tilde{d}_{1}+\tilde{d}_{2} - d_{1} = d_{2}.
\end{align*}
This completes the proof of the lemma.
\end{proof}

\section{Conservation of mass scenario}\label{john}

In this section we will establish Theorem \ref{d2d} under additional conservation of mass assumption
\begin{equation}\label{conserve}
\liminf_{k\to\infty} \delta_k =0, \qquad\text{where }\delta_k= \sum_{i=1}^{k}(d_{i}-\lambda_{i}).
\end{equation}
It is remarkable that we achieve this goal without assuming that the sequence $\{\lambda_i\}$ is unbounded. 
This requires a careful application of an infinite sequence of convex moves, also known as $T$-transforms \cite{kw}, to the original o.n.~basis of eigenvectors $\{f_i\}_{i\in\N}$. The key Lemma \ref{loss} guarantees that the limiting o.n.~sequence is complete.

\begin{lem}\label{loss} Let $\{f_{i}\}_{i\in\N}$ be an orthonormal set, and let $\{\alpha_{i}\}_{i\in\N}$ be a sequence in $[0,1]$. Set $\tilde{e}_{1}=f_{1}$ and inductively define for $i\in\N$,
\begin{equation}\label{los}
e_{i} =\sqrt{\alpha_{i}}\,\tilde{e}_{i} + \sqrt{1-\alpha_{i}}f_{i+1}\qquad\text{and}\qquad\tilde{e}_{i+1} =\sqrt{1-\alpha_{i}}\tilde{e}_{i} - \sqrt{\alpha_{i}}f_{i+1}.
\end{equation}
If for each $n\in \N$
\begin{equation}\label{loss0}
\prod_{i=n}^{\infty}(1-\alpha_{i})=0,\end{equation}
then $\{e_{i}\}_{i\in\N}$ is an orthonormal basis for $\overline{\lspan}\{f_{i}\}_{i\in\N}$ and \eqref{block2} holds. In particular,
if $\alpha_{i}<1$ for all $i$ and $\sum_{i=1}^{\infty}\frac{\alpha_{i}}{1-\alpha_{i}}=\infty$, then $\{e_{i}\}_{i\in\N}$ is an orthonormal basis for $\overline{\lspan}\{f_{i}\}_{i\in\N}$.
\end{lem}

\begin{proof} By induction, we see that for each $i\in\N$,
\[
\{e_1,e_2, \ldots, e_{i-1}, \tilde e_i, f_{i+1}, f_{i+2}, \ldots\}
\]
is an orthonormal sequence.
 Hence, $\{e_{i}\}_{i\in \N}$ is orthonormal and it is enough to show that $f_{j}\in\overline{\lspan}\{e_{i}\}_{i\in\N}$ for all $j\in\N$. Note that $e_{i}\in\lspan\{f_{j}\}_{j=1}^{i+1}$ and $\tilde{e}_{i}\in\lspan\{f_{j}\}_{j=1}^{i}$. Thus, $\langle f_{j},e_{i}\rangle = 0$ for $i\leq j-2$ and $\langle f_{j},\tilde{e}_{i}\rangle = 0$ for $i\leq j-1$. Also note that for each $n\in\N$ the sequence $\{e_{1},e_{2},\ldots,e_{n},\tilde{e}_{n+1}\}$ is an orthonormal basis for $\lspan\{f_{i}\}_{i=1}^{n+1}$. Thus, for $n\geq j-1$ we have
\begin{equation}\label{loss1}1-|\langle f_{j},\tilde{e}_{n+1}\rangle|^{2} = \sum_{i=1}^{n}|\langle f_{j},e_{i}\rangle|^{2}.\end{equation}

If we set $\alpha_{0}=1$, then
\[\langle f_{j},\tilde{e}_{j}\rangle = -\sqrt{\alpha_{j-1}}\]
for all $j\in \N$. For $n\geq 0$ we have
\[\langle f_{j},\tilde{e}_{j+n}\rangle = \sqrt{1-\alpha_{j+n-1}}\langle f_{j},\tilde{e}_{j+n-1}\rangle,\]
so that, by induction for $n\geq 0$ we have
\begin{equation}\label{loss2}\langle f_{j},\tilde{e}_{j+n}\rangle = -\left(\alpha_{j-1}\prod_{k=j}^{j+n-1}(1-\alpha_{j+k})\right)^{\frac{1}{2}}.\end{equation}
Letting $n\to\infty$ in \eqref{loss2}, we see from \eqref{loss0} that $\lim_{n\to\infty} \langle f_{j},\tilde{e}_{n}\rangle =0$. 
Hence, \eqref{loss1} implies that for each $j\in\N$
\[\sum_{i=1}^{\infty}|\langle f_{j},e_{i}\rangle|^{2} = 1,\]
That is, $f_{j}\in\overline{\lspan}\{e_{i}\}_{i\in\N}$, which completes the proof.

Finally, consider the case that $\alpha_{i}<1$ for all $i\in\N$, and $\sum_{i=1}^{\infty}\frac{\alpha_{i}}{1-\alpha_{i}}=\infty$. In this case we have
\[\sum_{i=n}^{k}\frac{\alpha_{i}}{1-\alpha_{i}}\leq \prod_{i=n}^{k}\Big(1+\frac{\alpha_{i}}{1-\alpha_{i}}\Big) = \frac{1}{\prod_{i=n}^{k}(1-\alpha_{i})}.\]
Letting $k\to\infty$ we obtain \eqref{loss0}.
\end{proof}

\begin{lem}\label{loglem} If $\{t_{n}\}$ is a positive nonincreasing sequence with limit zero, then
\[\sum_{n=1}^{\infty}\frac{t_{n}-t_{n+1}}{t_{n+1}}=\infty.\]
\end{lem}

\begin{proof} Since $(t_{n}-t_{n+1})/t_{n+1} = t_{n}/t_{n+1}-1$, we may assume $t_{n+1}/t_{n}\to 1$ as $n\to\infty$. Since $t_{n}/t_{n+1}\geq 1$ we have
\[\sum_{n=1}^{k}\frac{t_{n}-t_{n+1}}{t_{n+1}} = \sum_{n=1}^{k}\Big(\frac{t_{n}}{t_{n+1}}-1\Big) \geq \sum_{n=1}^{k}\log\Big(\frac{t_{n}}{t_{n+1}}\Big) = \log(t_{1})-\log(t_{k+1})\to \infty \quad\text{as }k\to\infty.\]

\end{proof}

Next, we prove the first preliminary version of Theorem \ref{d2d} under the additional assumption that $\{\delta_{k}\}$ is strictly decreasing to $0$. However, we do not assume in Lemma \ref{decdel} that $\{d_{i}\}$ is arranged in nondecreasing order. Also in all subsequent results in Section \ref{john} we do not assume that $\{\lambda_i\}$ is an unbounded sequence.

\begin{lem}\label{decdel} Let $\boldsymbol\lambda =\{\lambda_{i}\}_{i\in\N}$ be a nondecreasing  sequence. Let $E$ be a symmetric operator with diagonal $\boldsymbol\lambda$ as in Definition \ref{d21}. If $\mathbf d= \{d_{i}\}_{i\in\N}$ is a  sequence such that the following two properties hold:
\begin{equation}\label{decdel-1}\lambda_{1}\leq d_{n}<\lambda_{n}\quad \text{for all}\ n\geq 2,\end{equation} 
\begin{equation}\label{decdel0} d_{1}= \lambda_{1} + \sum_{i=2}^{\infty}(\lambda_{i}-d_{i})< \lambda_{2},\end{equation}
then $E$ has diagonal $\mathbf d$, which is finitely derived from $\boldsymbol \lambda$.
\end{lem}

\begin{proof} For each $n\in\N$ set
\[\tilde{\lambda}_{n} := d_{n} - \sum_{i=n+1}^{\infty}(\lambda_{i}-d_{i}) = \lambda_{n} -\sum_{i=n}^{\infty}(\lambda_{i}-d_{i}).
\]
From \eqref{decdel-1} for each $n\geq 2$, we have
\[\tilde{\lambda}_{n} < d_{n}<\lambda_{n}\leq \lambda_{n+1}.\]
From \eqref{decdel0} we have
\[\tilde{\lambda}_{1} = \lambda_{1}< d_{1}< \lambda_{2}.\]
Thus, for all $n\in\N$ we have 
\begin{equation}\label{3in}
\tilde \lambda_n<d_n,\tilde \lambda_{n+1}<\lambda_{n+1} \qquad\text{and}\qquad
\tilde \lambda_n+\lambda_{n+1}=d_n+\tilde \lambda_{n+1}.
\end{equation}
We conclude that for all $n\in\N$
\[
\tilde \alpha_{n} := \frac{\lambda_{n+1}-d_{n}}{\lambda_{n+1}-\tilde{\lambda}_{n}} = \frac{\lambda_{n+1}-d_{n}}{\lambda_{n+1}-d_{n} + \sum_{i=n+1}^{\infty}(\lambda_{i}-d_{i})}
\in(0,1).
\]
For each $n\in\N$ set
\[t_{n} = \sum_{i=n}^{\infty}(\lambda_{i}-d_{i}).\]
Note that $t_{1}=0$, and $\{t_{i}\}_{i=2}^{\infty}$ is a positive, nonincreasing sequence with limit zero. By Lemma \ref{loglem} we have
\begin{equation}\label{decdel1}
\sum_{n=1}^{\infty}\frac{\tilde \alpha_{n}}{1- \tilde \alpha_{n}} = \sum_{n=1}^{\infty}\frac{\lambda_{n+1}-d_{n}}{\sum_{i=n+1}^{\infty}(\lambda_{i}-d_{i})}\geq \sum_{n=1}^{\infty}\frac{\lambda_{n}-d_{n}}{\sum_{i=n+1}^{\infty}(\lambda_{i}-d_{i})} = \sum_{n=1}^{\infty}\frac{t_{n}-t_{n+1}}{t_{n+1}}=\infty.\end{equation}

Let $\{f_{n}\}_{n\in\N}$ be an orthonormal sequence with respect to which $E$ has diagonal $\boldsymbol \lambda$. We shall now define an orthonormal sequence
$\{e_{n}\}_{n\in\N}$ as in Lemma \ref{loss} for an appropriate choice of the sequence $\{\alpha_n\}_{n\in \N}$.

We have $\langle Ef_{1},f_{1}\rangle = \lambda_1=\tilde \lambda_1$, and $\langle Ef_{2},f_{2}\rangle = \lambda_2$. By Lemma \ref{offdiag} there exist $\alpha_1\in [\tilde\alpha_1,1)$ and $\theta_{2}\in[0,2\pi)$ such that vectors
\[
e_{1} = \sqrt{\alpha_{1}}f_{1} + \sqrt{1-\alpha_{1}}e^{i\theta_{2}}f_2 \quad\text{and}\quad \tilde{e}_2 = \sqrt{1-\alpha_{1}}f_{1} - \sqrt{\alpha_{1}}e^{i\theta_{2}}f_{2}
\]
form an orthonormal basis for $\lspan\{f_{1},f_{2}\}$ and
\[
\langle Ee_{1},e_{1}\rangle = d_{1}\quad\text{and}\quad \langle E\tilde{e}_{2},\tilde{e}_{2}\rangle = \tilde\lambda_2.
\]

Now, we may inductively assume that for some $n\ge 2$ we have an orthonormal basis $\{e_1,\ldots,e_{n-1},\tilde{e}_n\}$ for $\lspan\{f_{j}\}_{j=1}^n$ such that
\[
\langle Ee_j,e_j \rangle = d_j \quad\text{for }j\leq n-1 \qquad\text{and}\qquad 
\langle E\tilde{e}_n,\tilde{e}_n \rangle = \tilde \lambda_n.
\]
Using \eqref{3in}, by Lemma \ref{offdiag} there exist $\alpha_n\in [\tilde\alpha_n,1)$ and $\theta_{n+1}\in[0,2\pi)$ such that the vectors
\[
e_n = \sqrt{\alpha_n}\tilde{e}_n + \sqrt{1-\alpha_n}e^{i\theta_n}f_{n+1} \quad\text{and}\quad \tilde{e}_{n+1} = \sqrt{1-\alpha_n}\tilde{e}_n - \sqrt{\alpha_n}e^{i\theta_{n+1}}f_{n+1}
\]
form an orthonormal basis for $\lspan\{\tilde{e}_n,f_{n+1}\}$ and
\[
\langle Ee_n,e_n\rangle = d_n\quad\text{and}\quad \langle E\tilde{e}_{n+1},\tilde{e}_{n+1}\rangle = \tilde \lambda_{n+1}.
\]
The fact that $\alpha_n<1$ for all $n\in\N$ is a consequence of strict inequalities in \eqref{3in}. 

Observe that the above procedure yields an orthonormal sequence $\{e_n\}_{n=1}^{\infty}$ that is obtained by applying Lemma \ref{loss} to $\{e^{i\theta_n}f_n \}_{n\in\N}$ with $\{\alpha_n\}_{n\in\N}$ and $\{\theta_n\}_{n\in\N}$ as already defined and $\theta_1=0$. Since for all $n\in\N$, $\alpha_n\in [\tilde\alpha_n,1)$, by \eqref{decdel1} we have 
\[
\sum_{n=1}^\infty \frac{\alpha_n}{1-\alpha_n}
\ge
\sum_{n=1}^\infty \frac{\tilde \alpha_n}{1-\tilde \alpha_n}=\infty.
\]
Hence, by Lemma \ref{loss} $\{e_n\}_{n\in\N}$ is an orthonormal basis for $\overline{\lspan}\{f_n\}_{n\in\N}$. By \eqref{los} each vector $e_n$ is a linear combination $f_1,\ldots,f_{n+1}$. Therefore, $E$ has diagonal $\mathbf d$, which is finitely derived from $\boldsymbol \lambda$.
\end{proof}

The following is the second preliminary version of the main result of this section. The final result of this section, which is Theorem \ref{lim0}, will be identical with the exception of the extra assumption that $d_{1}<\lambda_{2}$.

\begin{lem}\label{lim0lem} Let $\mathbf d=\{d_{i}\}_{i\in\N}$ and $\boldsymbol\lambda=\{\lambda_{i}\}_{i\in\N}$ be nondecreasing  sequences such that \eqref{delta} and \eqref{conserve} hold. Let $E$ be a symmetric operator with diagonal $\boldsymbol\lambda$.
If $\lambda_{2}>d_{1}$, then $E$ has diagonal $\mathbf d$,
which is finitely derived from $\boldsymbol \lambda$.
\end{lem}

\begin{proof} By Theorem \ref{red} we may assume that $\delta_{k}>0$ for all $k\in\N$. Inductively define the sequence $\{m_{j}\}$ as follows. Set $m_{1}=1$ and for $j\geq 2$ set $m_{j}=\min\{n>m_{j-1}\colon \delta_{n}< \delta_{m_{j-1}}\}$.

For each $j\in\N$ and $i=m_{j}+1,m_{j+1},\ldots,m_{j+1}$ set
\[\tilde{d}_{i} = \frac{\delta_{m_{j+1}}-\delta_{m_{j}}}{m_{j+1}-m_{j}} + \lambda_{i}.\]
Also set $\tilde{d}_{1}=d_{1}$ and define
\[\tilde{\delta}_{k}:=\sum_{i=1}^{k}(\tilde{d}_{i}-\lambda_{i}).\]
By induction, for each $j\in\N$ and $k=m_{j}+1,\ldots,m_{j+1}$ we have
\begin{equation}\label{dti1}
\tilde{\delta}_{k} = \delta_{m_{j}} + \frac{\delta_{m_{j+1}}-\delta_{m_{j}}}{m_{j+1}-m_{j}}(k-m_{j}).
\end{equation}
In particular, we have
\begin{equation}\label{dti}
\tilde{\delta}_{m_{j}}=\delta_{m_{j}} \qquad\text{for all }
j\in\N.
\end{equation}
Since $\delta_{k}\ge \delta_{m_{j}}>\delta_{m_{j+1}}$ for all $m_j<k< m_{j+1}$, we have
\[\delta_{m_{j}}+\sum_{i=m_{j}+1}^{k}(\tilde{d}_{i}-\lambda_{i}) = \tilde{\delta}_{k} = \delta_{m_{j}} + \frac{\delta_{m_{j+1}}-\delta_{m_{j}}}{m_{j+1}-m_{j}}(k-m_{j}) \leq \delta_{m_{j}}\leq\delta_{k} = \delta_{m_{j}}+\sum_{i=m_{j}+1}^{k}(d_{i}-\lambda_{i}).\]
Combining this with \eqref{dti} shows that $\{d_{i}\}_{i=m_{j}+1}^{m_{j+1}}\preccurlyeq\{\tilde d_{i}\}_{i=m_{j}+1}^{m_{j+1}}$ for each $j\in\N$.

Using $\delta_{m_{j+1}}-\delta_{m_{j}}< 0$, \eqref{dti1}, and \eqref{dti} we deduce that the sequence $\{\tilde{\delta}_{k}\}$ is decreasing and $\lim_{k\to \infty}\tilde{\delta}_k=0$. Moreover, we have $\tilde{d}_{1} = d_{1} <\lambda_{2}$ and $\lambda_1 \le \tilde d_n$ for all $n\ge 2$. Applying Lemma \ref{decdel} to the sequences $\boldsymbol \lambda $ and $\tilde{\mathbf d} := \{\tilde{d}_{i}\}_{i\in \N}$ shows that $E$ has diagonal $\tilde{\mathbf d}$, which is finitely derived from $\boldsymbol \lambda$.

Finally, since the sets $I_{j}=\{m_{j},\ldots,m_{j+1}-1\}$ are disjoint, and $\{d_{i}\}_{i\in I_{j}}\preccurlyeq \{\tilde{d}_{i}\}_{i\in I_{j}}$, Lemma \ref{block} shows that $E$ has diagonal $\mathbf d$, which is finitely derived from $\tilde{\mathbf d}$, and hence finitely derived from $\boldsymbol \lambda$.
\end{proof}

\begin{lem}\label{flat} Let $\mathbf d = \{d_{i}\}_{i\in\N}$ and $\boldsymbol\lambda = \{\lambda_{i}\}_{i\in\N}$ be nondecreasing  sequences such that
such that \eqref{delta} and \eqref{conserve} hold. Let $E$ be a symmetric operator defined on a dense domain $\mathcal D$.
If the following two conditions hold:
\begin{enumerate}
\item there exists $N\in\N$ such that $\delta_{N}\leq \delta_{k}$ for all $k\leq N$,
\item $E$ has diagonal $\tilde{\mathbf d}:=\{\tilde{d}_{i}\}_{i\in\N}$, where
\begin{equation}\label{flat0}
\tilde{d}_{i}: = \begin{cases} \lambda_{1}+\delta_{N} & i=1,\\
 \lambda_{i} & i=2,\ldots,N,\\ 
 d_{i} & i>N,
 \end{cases}\end{equation}
\end{enumerate}
then $E$ has diagonal $\mathbf d$, which is finitely derived from $\tilde{\mathbf d}$.
\end{lem}

\begin{proof} Let $I_{1}=\{1,\ldots,N\}$. In light of Lemma \ref{block} it is enough to show that $\{d_{i}\}_{i\in I_{1}}\preccurlyeq\{\tilde{d}_{i}\}_{i\in I_{1}}$. Let $\{\tilde{d}_{i}^{\,\,\uparrow}\}_{i=1}^{N}$ denote the nondecreasing rearrangement of $\{\tilde{d}_{i}\}_{i=1}^{N}$, then for $k=1,\ldots,N$
\[\sum_{i=1}^{k}\tilde{d}_{i}^{\,\,\uparrow}\leq \sum_{i=1}^{k}\tilde{d}_{i} = \delta_{N} + \sum_{i=1}^{k}\lambda_{i} = \sum_{i=1}^{k}d_{i}+\delta_{N}-\delta_{k}\leq \sum_{i=1}^{k}d_{i}.\]
Together with the observation that both of the inequalities above become equality when $k=N$ demonstrates the desired majorization.
\end{proof}

We are now ready to show Theorem \ref{unb} under the additional hypothesis \eqref{conserve}, but without the assumption that $\{\lambda_i\}$ is unbounded.

\begin{thm}\label{lim0} Let $\mathbf d=\{d_{i}\}_{i\in\N}$ and $\boldsymbol\lambda=\{\lambda_{i}\}_{i\in\N}$ be nondecreasing  sequences such that \eqref{delta} and \eqref{conserve} hold. Let $E$ be a symmetric operator with diagonal $\boldsymbol\lambda$ as in Definition \ref{d21}. Then, $E$ has diagonal $\mathbf d$, which is finitely derived from $\boldsymbol \lambda$.
\end{thm}

\begin{proof} By Theorem \ref{red} we may assume that $\delta_{k}>0$ for all $k\in\N$. We also claim that $\boldsymbol\lambda$ is not a constant sequence. On the contrary, suppose $\lambda_{i}=L$ for all $i\in\N$. Since $\mathbf d$ is nondecreasing and $\liminf_{k\to\infty}\delta_{k}=0$ we conclude that $d_{i}\nearrow L$ as $i\to\infty$. The assumption that $\delta_{1}>0$ implies $d_{1}>L$, which is a contradiction.

Since $\boldsymbol \lambda$ is not constant, there is some $M\in\N$ such that $\lambda_{1}<\lambda_{M}$. Choose $N>M$ such that
\begin{equation}\label{lim01}\delta_{N}\leq \delta_{k}\quad\text{for all}\ k\leq N.\end{equation}
and $\delta_{N}<\lambda_{M}-\lambda_{1}$. Since $\lambda_{M}\leq \lambda_{N+1}$ we also have
\begin{equation}\label{lim02} \lambda_{N+1}>\delta_{N}+\lambda_{1}.\end{equation}

Define the sequence $\tilde {\mathbf d}=\{\tilde{d}_{i}\}_{i\in\N}$ as in \eqref{flat0}. Define the sequences $\{c_{i}\}$ and $\{\mu_{i}\}$ by 
\[c_{i}=\begin{cases}\tilde{d}_{1} & i=1,\\ \tilde{d}_{i+N-1} & i\geq 2,\end{cases}\qquad\text{and}\qquad \mu_{i}=\begin{cases}\lambda_{1} & i=1,\\ \lambda_{i+N-1} & i\geq 2.\end{cases}\]
Note that
\[\tilde{\delta}_{k}:=\sum_{i=1}^{k}(c_{i}-\mu_{i}) = \delta_{N+k-1} \qquad\text{for all }k\in\N.
\]
 From \eqref{lim02} we see that $c_{1}=\delta_{N}+\lambda_{1}< \lambda_{N+1}=\mu_{2}$. By our hypothesis, $E$ has diagonal $\{\mu_i\}$ with respect to o.n.~sequence $\{ f_i\}_{i=1,i>N}$.  Applying Lemma \ref{lim0lem} yields an o.n.~basis $\{\tilde e_i\}_{i=1,i>N}$ of $\overline{\lspan}\{ f_i\}_{i=1,i>N}$ with respect to which $E$ has diagonal $\{c_{i}\}$, which is finitely derived from $\{\mu_i\}$. Letting $\tilde e_i=f_i$ for $2\le i \le N$, yields an o.n.~sequence $\{\tilde e_i\}_{i\in \N}$ with respect to which $E$ has diagonal $\tilde {\mathbf d}$. By \eqref{lim01} we can apply Lemma \ref{flat} to obtain a desired o.n.~sequence $\{ e_i\}_{i\in \N}$, with respect to which $E$ has diagonal $\mathbf d$. Moreover, $\mathbf d$ is finitely derived from $\tilde {\mathbf d}$, and hence from $\boldsymbol \lambda$.
\end{proof}

\section{Mass vanishing at infinity scenario}

In this section we will show Theorem \ref{d2d} under complementary assumption to \eqref{conserve}. This involves a construction of an infinite sequence of convex moves continually transforming a diagonal sequence, where some of the mass must necessarily vanish at infinity. First we handle the strong domination case $\lambda_k\le d_k$ for every $k\in\N$. Equivalently, the sequence $\{\delta_k\}_{k\in\N}$ is assumed to be nondecreasing in Lemma \ref{nondec}.

\begin{lem}\label{nondec}
Let $E$ be a symmetric operator defined on a dense domain $\mathcal{D}$.  Let $\mathbf d=\{d_{i}\}_{i=1}^{\infty}$ and $\boldsymbol\lambda=\{\lambda_{i}\}_{i=1}^{\infty}$ be nondecreasing unbounded sequences with $d_i\ge \lambda_i$ for every $i$. If there exists an orthonormal sequence $\{f_{i}\}_{i\in\N}\subset\mathcal{D}$ such that
\[
\langle Ef_{i},f_{i}\rangle = \lambda_i
\qquad\text{for all } i\in \N,
\]
then there exists an orthonormal sequence $\{e_{i}\}_{i\in\N}$ satisfying \eqref{block2} and 
\[\langle Ee_{i},e_{i}\rangle = d_{i}
\qquad\text{for all } i\in \N.
\]
\end{lem}

\begin{proof}
Suppose that $I$ is an infinite subset of $\N$. For any such subset we define inductively an increasing sequence $\{i_{k}\}_{k=1}^{\infty}$ in $I$ by letting $i_1=\min I$ and choosing $i_k\in I$ large enough to have
    \begin{align}\label{eq:subseqassumption}
	\lambda_{i_{k}}>2d_{i_{k-1}} \qquad k\ge 2.
    \end{align}
In addition, we require that $I \setminus \{i_k: k\in\N\}$ is infinite.
    This is possible since the sequence $\{\lambda_i\}$ is not bounded. Now recursively define another sequence by $x_{i_1}=\lambda_{i_1}$ and
    \begin{align*}
	x_{i_{k+1}}=\lambda_{i_{k+1}}+x_{i_k}-d_{i_k} \qquad k\ge 1.
    \end{align*}
    Note that $x_{i_2}\le \lambda_{i_2}$ and $x_{i_2}-d_{i_1}>0$ (using condition \eqref{eq:subseqassumption}). By induction we get that for any $k\ge 1$
    \begin{equation}\label{4in}
        x_{i_k}\le d_{i_k}<x_{i_{k+1}}\le \lambda_{i_{k+1}}.
    \end{equation}
Furthermore,
\begin{equation}\label{2in}
\tilde \alpha_k:=\frac{\lambda_{i_{k+1}}-d_{i_k}}{\lambda_{i_{k+1}}-x_{i_k}}>\frac{\lambda_{i_{k+1}}/2}{\lambda_{i_{k+1}}}=\frac{1}{2} 
\end{equation}
    Now we are ready to start constructing an o.n.~sequence $\{e_{i_k}\}_{k=1}^\infty$.

    We have $\langle Ef_{i_1},f_{i_1}\rangle = x_{i_1}$, and $\langle Ef_{i_{2}},f_{i_{2}}\rangle = \lambda_{i_{2}}$. By Lemma \ref{offdiag} there exist $\alpha_1\in [\tilde\alpha_1,1]$ and $\theta_{2}\in[0,2\pi)$ such that vectors
\[
e_{i_1} = \sqrt{\alpha_{1}}f_{i_1} + \sqrt{1-\alpha_{1}}e^{i\theta_{2}}f_{i_{2}} \quad\text{and}\quad \tilde{e}_{i_{2}} = \sqrt{1-\alpha_{1}}f_{i_1} - \sqrt{\alpha_{1}}e^{i\theta_{2}}f_{i_{2}}
\]
form an orthonormal basis for $\lspan\{f_{i_1},f_{i_{2}}\}$ and
\[
\langle Ee_{i_1},e_{i_1}\rangle = d_{i_1}\quad\text{and}\quad \langle E\tilde{e}_{i_{2}},\tilde{e}_{i_{2}}\rangle = x_{i_2}.\]

Now, we may inductively assume that  for some $k\ge 2$ we have an orthonormal basis $\{e_{i_{1}},\ldots,e_{i_{k-1}},\tilde{e}_{i_{k}}\}$ for $\lspan\{f_{i_{j}}\}_{j=1}^{k}$ such that
\[
\langle Ee_{i_{j}},e_{i_{j}}\rangle = d_{i_{j}}\quad\text{for }j\leq k-1 \qquad\text{and}\qquad 
\langle E\tilde{e}_{i_{k}},\tilde{e}_{i_{k}}\rangle = x_{i_{k}}.
\]
Using \eqref{4in}, by Lemma \ref{offdiag} there exist $\alpha_k\in [\tilde\alpha_k,1]$ and $\theta_{k+1}\in[0,2\pi)$ such that the vectors
\[
e_{i_{k}} = \sqrt{\alpha_{k}}\tilde{e}_{i_{k}} + \sqrt{1-\alpha_{k}}e^{i\theta_{k}}f_{i_{k+1}} \quad\text{and}\quad \tilde{e}_{i_{k+1}} = \sqrt{1-\alpha_{k}}\tilde{e}_{i_{k}} - \sqrt{\alpha_{k}}e^{i\theta_{k+1}}f_{i_{k+1}}
\]
form an orthonormal basis for $\lspan\{\tilde{e}_{i_{k}},f_{i_{k+1}}\}$ and
\[\langle Ee_{i_{k}},e_{i_{k}}\rangle = d_{i_{k}}\quad\text{and}\quad \langle E\tilde{e}_{i_{k+1}},\tilde{e}_{i_{k+1}}\rangle = x_{i_{k+1}}.\]
This completes the inductive step, and thus we have an orthonormal sequence $\{e_{i_{k}}\}_{k=1}^{\infty}$. 

Observe that this is exactly the orthonormal sequence obtained by applying Lemma \ref{loss} to $\{e^{i\theta_{k}}f_{i_{k}}\}_{k\in\N}$ with $\{\alpha_{k}\}_{k\in\N}$ and $\{\theta_{k}\}_{k\in\N}$ as already defined with $\theta_1=0$. By \eqref{2in} we have $\alpha_{k}>1/2$ for all $k\in\N$. Hence, by Lemma \ref{loss} $\{e_i\}_{i\in I_1}$ is an orthonormal basis for $\mathcal H_1=\overline{\lspan}\{f_{i}\}_{i\in I_1}$, with respect to which $E$ has diagonal $\{d_{i}\}_{i\in I_1}$, where $I_1=\{i_k:k\in \N\}$. Moreover, diagonal $\{d_{i}\}_{i\in I_1}$ is finitely derived from $\{\lambda_{i}\}_{i\in I_1}$.

In the initial step we run the above construction starting with the full index set $I=\N$ to obtain the required diagonal subsequence indexed by $I_1$. Then, we repeat the above construction inductively with respect to the unused index set $I=\N \setminus (I_1\cup \ldots \cup I_{k-1})$, $k\ge 2$, to obtain the required diagonal subsequence indexed by $I_k$. Since we always include the smallest unused element in $I$ and we leave out infinitely many unused indices, the family $\{I_k\}_{k\in\N}$ is a partition of $\N$. Thus, we obtain an orthogonal decomposition 
\[
\overline{\lspan}\{f_i\}_{i\in\N} = \bigoplus_{k=1}^\infty \mathcal H_k, \qquad\text{where }\mathcal H_k=\overline{\lspan}\{f_{i}\}_{i\in I_k}.
\]
For each subspace $\mathcal H_k$ we have constructed an orthonormal basis $\{e_{i}\}_{i\in I_k}$, with respect to which $E$ has diagonal $\{d_{i}\}_{i\in I_k}$, that is finitely derived from $\{\lambda_{i}\}_{i\in I_k}$. This defines the required orthonormal basis $\{e_i\}_{i\in \N}$ of $\overline{\lspan}\{f_i\}_{i\in\N}$ with respect to which $E$ has diagonal $\mathbf d$.
\end{proof}

We are now ready to show Theorem \ref{d2d} in the case when sequence $\{\delta_k\}_{k\in\N} $ as in \eqref{delta}, eventually stays above its $\liminf_{k\to\infty}\delta_k$.

\begin{thm}\label{limalpha} 
Let $E$ be a symmetric operator defined on a dense domain $\mathcal{D}$. 
Let $\mathbf d=\{d_{i}\}_{i=1}^{\infty}$ and $\boldsymbol\lambda=\{\lambda_{i}\}_{i=1}^{\infty}$ be nondecreasing  unbounded sequences such that \eqref{delta} holds. 
Assume that there exists $M\geq 0$ such that
\[
\delta_{k}\geq \alpha:=\liminf_{i\to\infty}\delta_i \qquad\text{for all }k\geq M.
\]
If $\boldsymbol\lambda$ is a diagonal of $E$, then $\mathbf d$ is a finitely derived diagonal of $E$.
\end{thm}

\begin{proof} By Lemma \ref{red} we may assume $\delta_{k}>0$ for all $k\in\N$. Fix $N\in\N$ such that $N> \max_{k\leq M-1}\{\frac{k\alpha}{\delta_{k}},M\}$. Hence,
\[
\delta_{k}\geq \frac{k\alpha}{N} \qquad\text{for }
k\leq M-1.
\]
Define
\[\tilde{d}_{i} = \begin{cases} d_{i} - \frac{\alpha}{N} & i=1,\ldots,N,
\\ d_{i} & i\geq N+1.
\end{cases}
\]
Observe that
\[\sum_{i=1}^{k}(\tilde{d}_{i} - \lambda_{i}) = \begin{cases}\delta_{k} - \frac{k\alpha}{N}\geq 0 & k\leq M-1,\\ 
\delta_{k} - \frac{k\alpha}{N}\geq \alpha-\frac{k\alpha}{N}\geq 0 & M\leq k\leq N,\\
\delta_{k}-\alpha \geq0 & k\geq N+1.
\end{cases}\]
The last equation implies that $\liminf_{k\to\infty}\sum_{i=1}^{k}(\tilde{d}_{i}-\lambda_{i})=0$. We may apply Theorem \ref{lim0} to deduce that $E$ has diagonal $\{\tilde{d}_{i}\}_{i\in\N}$, which is finitely derived from $\boldsymbol \lambda$. Since $d_i \ge \tilde d_i$ for all $i\in\N$, Lemma \ref{nondec} yields the desired diagonal $\{{d}_{i}\}_{i\in\N}$.
\end{proof}

Finally, we are left we the case when the sequence $\{\delta_k\}_{k\in\N}$ dips infinitely many times below its $\liminf_{k\to\infty}\delta_k$.

\begin{thm}\label{tonondec}
Let $E$ be a symmetric operator defined on a dense domain $\mathcal{D}$.
Let $\mathbf d=\{d_{i}\}_{i\in\N}$ and $\boldsymbol\lambda=\{\lambda_{i}\}_{i=1\in \N}$ be nondecreasing unbounded sequences such that \eqref{delta} holds. Assume that 
\[
\delta_{k}< \alpha:= \liminf_{i\to\infty} \delta_i
\qquad\text{for infinitely many }k.
\]
If $\boldsymbol\lambda$ is a diagonal of $E$, then $\mathbf d$ is a finitely derived diagonal of $E$. 
\end{thm}

\begin{proof}
We define inductively the index sequence $\{m_j\}_{j=0}^\infty$ as follows. Let $m_0=0$. For $j\ge 1$ set 
\[
m_j = \min\{ n>m_{j-1}: \forall k \ge n\quad \delta_n \le \delta_k  \}.
\]
That is, the sequence $\{m_j\}$ records consecutive global minima of the tail $\{\delta_n\}_{n>m_{j-1}}$. In particular, using the convention that $\delta_0=0$, we have
\begin{equation}\label{tnd3}
\delta_{m_{j-1}} \le \delta_{m_j} \le \delta_k
\qquad\text{for all } m_{j-1} < k \le m_j,
\quad j\ge 1.
\end{equation} 
Define the sequence $\{\tilde d_i \}_{i\in \N}$ by
\begin{equation}\label{tnd4}
\tilde d_i = \begin{cases} 
\lambda_i + (\delta_{m_j} - \delta_{m_{j-1}}) & \text{for }i=m_j, \ j \ge 1,\\
\lambda_i & \text{otherwise.}
\end{cases}
\end{equation}
Set 
\[
\tilde \delta_{k}=\sum_{i=1}^{k}(\tilde d_{i}-\lambda_{i}).
\]
For $j\ge 1$, set $I_j=\{m_{j-1}+1,\ldots,m_j\}$.
By \eqref{tnd3} and \eqref{tnd4}, for any $k\in I_j$ we have
\[
\delta_{m_{j-1}}+ \sum_{i=m_{j-1}+1}^k (d_i-\lambda_i) = \delta_k  \ge \delta_{m_{j+1}} \ge \delta_{m_{j-1}}+ \sum_{i=m_{j-1}+1}^k (\tilde d_i-\lambda_i)
\]
with equalities when $k=m_j$. This shows that $\{d_i\}_{i\in I_j} \preccurlyeq \{\tilde d_i\}_{i\in I_j}$. Since the sets $\{I_j\}_{j\in\N}$ form a partition of $\N$, we can apply Lemma \ref{block} to reduce the problem to showing that $E$ has diagonal $\{\tilde d_{i}\}_{i\in \N}$. This case is already covered by Lemma \ref{nondec} since the sequence $\{\tilde \delta_i\}_{i\in\N}$ is nondecreasing.
\end{proof}

Theorem \ref{d2d} now follows immediately by combining Theorems \ref{limalpha} and \ref{tonondec}.

\section{Remarks and Examples}\label{S5}

\subsection{Diagonals and eigenvalues of inverse operators} It is worth observing how our main result, Theorem \ref{unb}, is related to the result of Kaftal and Weiss \cite{kw} who characterized the diagonals of positive compact operators. The earlier result of Arveson and Kadison \cite{ak} characterized diagonals of positive trace class operators. In the case of positive compact operators that are not trace class, the trace
condition is not present both in \cite{kw} and in Theorem \ref{unb}. Hence, one might attempt to deduce Theorem \ref{unb} from \cite{kw}. 

For simplicity assume that the first eigenvalue of $E$ is $\la_1>0$. Then, the inverse $E^{-1}$ is a compact positive operator with eigenvalues $1/\la_1 \ge 1/\la_2 \ge \ldots \searrow 0$. Conversely, the inverse of positive self-adjoint operator with trivial kernel is unbounded with discrete spectrum. However, the diagonal does not behave in such controlled way when taking inverses. Thus, Theorem \ref{unb} does not follow from \cite{kw} in any obvious way. For the converse direction, Theorem \ref{lim0} implies a special case of \cite{kw} when $\liminf_{k\to\infty} \delta_k=0$.
Nevertheless, it is possible to deduce majorization for sums of inverses from the majorization of sums of eigenvalues as follows. 

We say that a sequence $\{a_i\}_{i\in \N}$ is (weakly) majorized by a sequence $\{b_i\}_{i\in \N}$, and write $\{a_i\}\prec\{b_i\}$, if
\[
\sum_{i=1}^n a_i \le \sum_{i=1}^n b_i \qquad\text{for all }n\in\N.
\]
Note that unlike (strong) majorization order $\preccurlyeq$, we do not alter the order of elements of the sequences. Recall the classical Hardy-Littlewood-P\'olya majorization theorem \cite[\S3.17]{HLP}.

\begin{thm}[Hardy-Littlewood-P\'olya majorization]
    Assume that $\{a_i\}$ and $\{b_i\}$ are nondecreasing sequences of positive real numbers such that $\{a_i\}\prec\{b_i\}$. Then for any concave increasing function $\Phi: \R_+ \to \R$ we have $\{\Phi(a_i)\} \prec \{\Phi(b_i)\} $. Similarly, when $\{a_i\}$ and $\{b_i\}$ are nonincreasing, then the result holds for convex increasing functions $\Phi$. 
\end{thm}

Let $a_i=\lambda_i$ and $b_i=d_i$ with the sequences coming from  the unbounded operator $E$ as in Theorem \ref{unb}. Now choose $\Phi(x)=-1/x$ to get that $\{1/d_i\}\prec \{1/\lambda_i\}$. Therefore, whenever $\{d_i\}$ is a possible diagonal for $E$, the sequence of inverses is a valid diagonal for the compact operator $E^{-1}$. Interestingly, the inverse procedure does not work. Even if $\{\tilde d_i\}$ is majorized by $\{1/\lambda_i\}$, the sequence of inverses $\{1/\tilde d_i\}$ does not need to majorize $\{\lambda_i\}$, since $\Phi(x)=-1/x$ is not convex.

As another consequence of Hardy-Littlewood-P\'olya majorization we get that whenever $\{d_i\}$ is a valid diagonal for $E$, the sequence of eigenvalues $\{e^{-\lambda_i t}\}$ of the heat operator $e^{-tE}$ majorizes $\{e^{-d_i t}\}$. Therefore the heat operator associated with $E$ admits diagonal $\{e^{-d_i t}\}$.

\subsection{Examples using Laplacians}
\label{sec:ExamplesUsingLaplacians}

Elliptic differential operators provide a broad and interesting class of operators falling into the scope of this paper. In particular, Laplace operators on domains $\Omega \subset \R^d$ imposed with various boundary conditions  can be closed in $L^2(\Omega)$ leading to essentially self-adjoint operators with discrete spectrum. This follows from classical considerations involving compactness of their inverses and compactness of the Sobolev embeddings. For more details see Bandle \cite{Ban80} or Blanchard-Br{\"u}ning \cite{BB92}.

To be more specific, consider two Laplace operators defined (weakly) on Sobolev spaces, via the corresponding quadratic forms:
\begin{itemize}
    \item Neumann Laplacian $\Delta_N$: domain $H^1(\Omega)$, quadratic form $\langle \Delta_N u,v\rangle= \int_\Omega \nabla u\cdot\nabla v \,dA$;
    \item Dirichlet Laplacian $\Delta_D$: domain $H^1_0(\Omega)$, quadratic form $\int_\Omega \nabla u\cdot\nabla v \,dA$.
\end{itemize}
It turns out that the eigenfunctions for these operators satisfy appropriate classical boundary conditions: Neumann $\partial_n u=0$ on $\partial\Omega$, and Dirichlet $u=0$ on $\partial\Omega$, respectively. See Chapters 5 and 6 of Laugesen \cite{LNotes} for a nice overview.

Let $\mu_j$ and $\lambda_j$ denote the eigenvalues (in nondecreasing order) for the Neumann and Dirichlet Laplacians, resp. It is easy to see (via operator domain inclusion) that for any $j$ we have $\mu_j\le\lambda_j$, see \cite{Ban80} or \cite[Chapter 10]{LNotes}. Therefore we have two sequences exhibiting strong domination as in Lemma \ref{nondec}.

Interestingly, these operator are not self-adjoint, or even symmetric, according to the theory of unbounded operators. They are defined on a dense subspace $H^1(\Omega)$ of $L^2(\Omega)$, however their adjoints have much smaller domain. One can however consider the same operators restricted to $H^2(\Omega)$. Assuming that $\Omega$ is somewhat smooth (locally Lipschitz boundary is enough), elliptic regularity theory implies that domain of the adjoint is now the same as for the operator. Hence we get self-adjoint operators on $H^2(\Omega)$ which agree with the weak formulation on their domains. See \cite[Chapters 18 and 19]{LNotes} for a detailed exposition.

\subsubsection{Dirichlet eigenvalues and Neumann Laplacian}
\label{sec:DirichletinNeumann}
We can ask for an orthonormal basis of $L^2(\Omega)$ such that the diagonal entries of the Neumann Laplacian equal the Dirichlet eigenvalues $\lambda_j\ge \mu_j$. Theorem \ref{unb} asserts that such a basis must exist. 

In the simplest possible case of an interval, $\Omega=[0,\pi]$, the Dirichlet eigenfunctions equal $\{u_j=\sin(j x)\}_{j\ge 1}$ and they form an o.n.~basis of $L^2$. These functions certainly belong to $H^1(\Omega)$ (or even $H^2(\Omega)$), so we already have the required o.n.~basis for $L^2(\Omega)$ (Fourier sine series). However, we are acting on these functions using Neumann Laplacian. This is irrelevant for the quadratic form definition, but the pointwise action is not simply the second derivative. In order to compute the Neumann Laplacian of $\sin(j x)$ we must first find the Fourier cosine series expansion of that function, since $\{\cos(j x)\}_{j \ge 0}$ is the o.n.~basis formed by the eigenfunction of the Neumann Laplacian. Therefore our transformations amount to constructing a cosine series for sine functions.

\subsubsection{Domain monotonicity for Dirichlet Laplacian}
\label{sec:DomainMonotonicityDirichlet}

It is also easy to see that if $\Omega_1\subset\Omega_2$, then $\lambda_j(\Omega_1)\ge\lambda_j(\Omega_2)$, simply because $H^1_0(\Omega_1)\subset H^1_0(\Omega_2)$ (by setting functions equal $0$ outside). If $\Omega_1$ is a relatively compact subset of $\Omega_2$ then the eigenfunctions of the Dirichlet Laplacian on $\Omega_1$ are concentrated on a compact subset of $\Omega_2$, hence they cannot form an o.n.~basis for $L^2(\Omega_2)$. Theorem \ref{unb} still asserts that there is an o.n.~basis of $L^2(\Omega_2)$ such that the diagonal of the Dirichlet Laplacian on $\Omega_2$ equals $\{\lambda_j(\Omega_1)\}$. However, it is not at all clear how to find such a basis.

\end{document}